\documentclass[11pt]{amsart}

\usepackage[a4paper,hmargin=3.5cm,vmargin=4cm]{geometry}
\usepackage{amsfonts,amssymb,amscd,amstext,verbatim}

\usepackage{fancyhdr}
\input xy
\xyoption{all}

\usepackage{times}

\usepackage{enumerate}
\usepackage{titlesec}
\usepackage{mathrsfs}

\titleformat{\section}
{\filcenter\bfseries\large} {\thesection{.}}{0.2cm}{}
\titleformat{\subsection}[runin]
{\bfseries} {\thesubsection{.}}{0.15cm}{}[.]
\titleformat{\subsubsection}[runin]
{\em}{\thesubsubsection{.}}{0.15cm}{}[.]

\usepackage[up,bf]{caption}

\newtheorem{theorem}{Theorem}[section]
\newtheorem{proposition}[theorem]{Proposition}

\newtheorem{lemma}[theorem]{Lemma}
\newtheorem{corollary}[theorem]{Corollary}

\theoremstyle{definition}

\newtheorem{remark}[theorem]{Remark}

\numberwithin{equation}{section}
\numberwithin{figure}{section}

\newcommand\Bcal{\mathcal{B}}

\newcommand\Lcal{\mathcal{L}}

\newcommand\Pcal{\mathcal{P}}

\newcommand\Ascr{\mathscr{A}}

\newcommand\Cscr{\mathscr{C}}

\newcommand\Iscr{\mathscr{I}}

\newcommand\Lscr{\mathscr{L}}
\newcommand\Oscr{\mathscr{O}}

\newcommand\C{\mathbb{C}}
\newcommand\D{\mathbb{D}}

\newcommand\N{\mathbb{N}}
\newcommand\R{\mathbb{R}}
\renewcommand\S{\mathbb{S}}

\newcommand\Z{\mathbb{Z}}

\renewcommand\span{\mathrm{span}}

\newcommand\wt{\widetilde}

\newcommand\di{\partial}

\newcommand\hra{\hookrightarrow}
\newcommand\longhookrightarrow{\ensuremath{\lhook\joinrel\relbar\joinrel\rightarrow}}
\newcommand\lra{\longrightarrow}

\newcommand\supp{\mathrm{supp}}
\newcommand\Qcal{\mathcal{Q}}

\begin{document}

\fancyhead[LO]{The Oka principle for holomorphic Legendrian curves in $\C^{2n+1}$}
\fancyhead[RE]{Franc Forstneri\v c and Finnur L\'arusson}
\fancyhead[RO,LE]{\thepage}

\thispagestyle{empty}

\vspace*{7mm}
\begin{center}
{\bf \LARGE The Oka principle \\  \vspace{2mm} for holomorphic Legendrian curves in $\C^{2n+1}$} 
\vspace*{5mm}

{\large\bf Franc Forstneri\v c and Finnur L\'arusson}
\end{center}

\vspace*{7mm}

\begin{quote}
{\small
\noindent {\bf Abstract}\hspace*{0.1cm}
Let  $M$ be a connected open Riemann surface. We prove that the space $\Lscr(M,\C^{2n+1})$ of all holomorphic Legendrian immersions of $M$ to $\C^{2n+1}$, $n\geq 1$, endowed with the standard holomorphic contact structure, 
is weakly homotopy equivalent to the space $\Cscr(M,\S^{4n-1})$ of continuous maps from $M$
to the sphere $\S^{4n-1}$. If $M$ has finite topological type, then these spaces  are homotopy equivalent.  We determine the homotopy groups of $\Lscr(M,\C^{2n+1})$ in terms of the homotopy groups of $\S^{4n-1}$.
It follows that $\Lscr(M,\C^{2n+1})$ is $(4n-3)$-connected. 

\vspace*{0.1cm}
\noindent{\bf Keywords}\hspace*{0.1cm} Riemann surface, Legendrian curve, Oka principle, absolute neighborhood retract 

\vspace*{0.1cm}

\noindent{\bf MSC (2010)}\hspace*{0.1cm} 53D10, 32E30, 32H02, 57R17
\vspace*{0.1cm}

\noindent{\bf Date}\hspace*{0.1cm} 6 November 2016; this version 15 May 2017}
\end{quote}

%
%
%
%
\section{Introduction}\label{sec:intro}
It is an interesting and important problem to describe the rough shape of mapping spaces 
that arise in analysis and geometry. Answering such a question typically amounts
to proving a {\em homotopy principle} (h-principle) to the effect that analytic solutions
can be classified by topological data; in particular, a solution exists in the absence of topological obstructions.
For a survey of the h-principle and its applications, see  
the monographs by Gromov \cite{Gromov1986book}, Eliashberg and Mishachev
\cite{EliashbergMishachev2002}, and Spring \cite{Spring2010}. In complex analysis,
a synonym for h-principle is {\em Oka principle}. This is a subject
with a long and rich history going back to Oka's paper \cite{Oka1939} 
in 1939; we refer to the monograph \cite{Forstneric2011}. 

In this paper, we describe the rough shape of the space $\Lscr(M,\C^{2n+1})$ of 
holomorphic Legendrian immersions of an open Riemann surface $M$ into the complex 
Euclidean space $\C^{2n+1}$, $n\geq 1$, with the standard holomorphic contact structure \eqref{eq:alpha}. 
Our main result is that $\Lscr(M,\C^{2n+1})$ is weakly homotopy equivalent to the space
$\Cscr(M,\S^{4n-1})$ of continuous maps from $M$ to the $(4n-1)$-dimensional sphere,
and is homotopy equivalent to it if $M$ has finite topological type;
see Corollary \ref{cor:whe}. Analogous results for several other mappings spaces
were obtained in \cite{ForstnericLarusson2016}.

We begin by introducing the relevant spaces of maps. All spaces under
consideration are endowed with the compact-open topology, 
unless otherwise specified.

A holomorphic $1$-form $\alpha$ on a complex manifold $X$ of odd dimension $2n+1\ge 3$ 
is said to be a {\em contact form} if it satisfies the nondegeneracy condition
$\alpha \wedge(d\alpha)^n \neq 0$ at every point of $X$. 
The model is the complex Euclidean space $\C^{2n+1}$ with coordinates
\begin{equation}\label{eq:coord}
	x=(x_1,\ldots,x_n)\in\C^n,\quad y=(y_1,\ldots,y_n)\in\C^n, \quad z\in\C,
\end{equation}
and $\alpha$ the standard contact form
\begin{equation}\label{eq:alpha}
	\alpha = dz + \sum_{j=1}^n x_j \, dy_j.
\end{equation}
By Darboux's theorem, every holomorphic contact form on a $(2n+1)$-dimensional complex manifold
is given by \eqref{eq:alpha} in some local holomorphic coordinates at each point
(see \cite[Theorem A.2]{AlarconForstnericLopez2016Legendrian}; 
for the smooth case, see e.g.\ \cite[Theorem 2.5.1]{Geiges2008CUP}).

A smooth map $F\colon M\to \C^{2n+1}$ from a smooth manifold $M$
is said to be {\em Legendrian} if $F^*\alpha=0$ on $M$. It is an elementary observation 
that every smooth Legendrian surface in a $3$-dimensional complex contact manifold is a 
complex curve; see Proposition \ref{prop:complex}.

Let $M$ be a connected open Riemann surface.  Denote by $\Iscr(M,\C^{2n})$ the space of all holomorphic immersions $M\to\C^{2n}$, and consider the closed subspace
\[
	\Iscr_*(M,\C^{2n}) = \bigl\{ (x,y)\in \Iscr(M,\C^{2n}) : xdy=  \sum_{j=1}^n x_j \, dy_j \ \ 
	\text{is an exact $1$-form on}\ M\bigr\}.
\]
Elements of $\Iscr_*(M,\C^{2n})$ will be called {\em exact holomorphic immersions}. Let
\begin{equation}\label{eq:incl}
	 \Iscr_*(M,\C^{2n}) \stackrel{\phi}{\longhookrightarrow}   \Iscr(M,\C^{2n})
\end{equation}
be the inclusion. Note that the map 
\[
	\Lscr(M,\C^{2n+1}) \longrightarrow  \Iscr_*(M,\C^{2n}) \times\C,
\]
given for a fixed choice of a base point $u_0\in M$ by 
\begin{equation}\label{eq:homeo}
	\Lscr(M,\C^{2n+1}) \ni (x,y,z) \longmapsto (x,y,z(u_0)) \in \Iscr_*(M,\C^{2n}) \times\C,
\end{equation}
is a homeomorphism. This follows immediately from the formula 
\begin{equation}\label{eq:z-component}
	z(u)=z(u_0)-\int_{u_0}^u xdy,\quad u\in M,
\end{equation}
which holds for any Legendrian immersion
$(x,y,z)\in \Lscr(M,\C^{2n+1})$, observing also that the integral $\int_{u_0}^u xdy$ is independent 
of the choice of a path from $u_0$ to $u$ (and hence defines a Legendrian immersion 
by the above formula) if and only if $(x,y)\in  \Iscr_*(M,\C^{2n})$. It follows 
that the projection $\pi\colon \Lscr(M,\C^{2n+1}) \to \Iscr_*(M,\C^{2n})$ is a homotopy equivalence.

Fix a nowhere vanishing holomorphic $1$-form $\theta$ on $M$; such exists by the 
Oka-Grauert principle \cite[Theorem 5.3.1]{Forstneric2011}. The specific  choice of $\theta$
will be irrelevant. For every immersion $\sigma \in \Iscr(M,\C^{2n})$, the map
$d\sigma/\theta\colon M\to \C^{2n}$ is holomorphic and it avoids the origin $0\in\C^{2n}$. 
The correspondence $\sigma \mapsto d\sigma/\theta$  defines a continuous map
\[  \varphi : \Iscr(M,\C^{2n}) \longrightarrow \Oscr(M,\C^{2n}_*). \]
Here, $\C^{2n}_*=\C^{2n}\setminus\{0\}$.  By \cite[Theorem 1.4]{ForstnericLarusson2016},  $\varphi$ is a weak 
homotopy equivalence, and a homotopy equivalence if $M$ has finite topological type. 

Let $\iota\colon\Oscr(M,\C^{2n}_*)\hookrightarrow \Cscr(M,\C^{2n}_*)$ denote the
inclusion of the space of holomorphic maps into the space of continuous maps. 
Since $\C^{2n}_*$ is a homogeneous space of the complex Lie group $GL_{2n}(\C)$,
$\iota$ is a weak homotopy equivalence by the Oka-Grauert principle \cite[Theorem 5.3.2]{Forstneric2011};  
if $M$ has finite topological type, then $\iota$ is a homotopy equivalence \cite{Larusson2015PAMS}.

Finally, the radial projection $\C^{2n}_*\to \S^{4n-1}$  onto the unit sphere 
induces a homotopy equivalence $\tau\colon \Cscr(M,\C^{2n}_*)\to \Cscr(M,\S^{4n-1})$. 

In summary, all the maps in the following sequence except $\phi$
are known to be weak homotopy equivalences, and to be homotopy 
equivalences when $M$ has finite topological type:
\begin{multline} \label{eq:fivemaps}
	\Lscr(M,\C^{2n+1})  \stackrel{\pi}{\longrightarrow} \Iscr_*(M,\C^{2n}) 
	\stackrel{\phi}{\longhookrightarrow} \Iscr(M,\C^{2n}) 
	 \stackrel{\varphi}{\longrightarrow} \\
	 \stackrel{\varphi}{\longrightarrow}  \Oscr(M,\C^{2n}_*) 
	\stackrel{\iota}{\longhookrightarrow} \Cscr(M,\C^{2n}_*) \stackrel{\tau}{\longrightarrow} \Cscr(M,\S^{4n-1}).
\end{multline}

The following is our main result.

%
%
\begin{theorem}\label{th:immersions}
For every connected open Riemann surface $M$, the inclusion 
\[
	\Iscr_*(M,\C^{2n}) \longhookrightarrow \Iscr(M,\C^{2n})
\]
of the space of exact holomorphic immersions $M\to\C^{2n}$, $n\geq 1$, into the space of all holomorphic immersions
is a weak homotopy equivalence, and a homotopy equivalence if the surface $M$ has finite topological type.
\end{theorem}

Since a composition of (weak) homotopy equivalences is again a (weak) homotopy equivalence,
Theorem \ref{th:immersions} implies the following.

%
%
\begin{corollary}\label{cor:whe}
All the maps in the sequence \eqref{eq:fivemaps}, and compositions thereof, are
weak homotopy equivalences, and homotopy equivalences if $M$ has finite topological type.
This holds in particular for the map $\Lscr(M,\C^{2n+1})  \to  \Cscr(M,\S^{4n-1})$.
\end{corollary}

The first part of Theorem \ref{th:immersions} follows immediately from Theorem \ref{th:parametric}, which establishes the 
parametric Oka principle with approximation for the inclusion \eqref{eq:incl}. The same proof
gives the parametric Oka principle with approximation 
for holomorphic Legendrian immersions;  see Remark \ref{rem:php-whe}.  
The basic case of the latter result is
\cite[Theorem 1.1]{AlarconForstnericLopez2016Legendrian}. The parametric
case considered here is more demanding, but unavoidable when analysing
the homotopy type of these mapping spaces. The second part of Theorem \ref{th:immersions} is proved in Sec.\ \ref{sec:strong}. Our proofs bring together tools from 
complex analysis and geometry, convex integration theory, and the theory of absolute neighborhood retracts.  

The examples in \cite{Forstneric2016hyp}
show that Theorem \ref{th:immersions} and Corollary \ref{cor:whe} have no analogue for more general 
holomorphic contact structures on Euclidean spaces; see Remark \ref{rem:hyperbolic}.
In those examples, the contact structure is Kobayashi hyperbolic, and hence it does not admit
any nonconstant Legendrian maps from $\C$ or $\C_*$.

It was shown in \cite{AlarconForstnericLopez2016Legendrian} that the space $\Lscr(M,\C^{2n+1})$ 
is very big from the analytic viewpoint. In particular, 
every holomorphic Legendrian map $K\to \C^{2n+1}$ from a (neighborhood of) a 
compact $\Oscr(M)$-convex subset  $K\subset M$ can be approximated on $K$
by proper holomorphic Legendrian embeddings of $M$ into $\C^{2n+1}$.  
Furthermore, every bordered Riemann surface carries a
{\em complete} proper holomorphic Legendrian embedding into the ball of $\C^{3}$,
and a complete bounded holomorphic Legendrian embedding in $\C^3$ 
such that the image surface is bounded by Jordan curves. 
(An immersion $F\colon M\to\R^{n}$ is said to be complete if the pull-back of the Euclidean metric 
on $\R^n$ by $F$ is a complete metric on $M$.)
Analogous results for holomorphic immersions $M\to\C^n$ $(n\ge 2)$,
null holomorphic curves in $\C^n$ $(n\ge 3)$, and conformal minimal immersions in
$\R^n$ ($n\ge 3$) were proved in \cite{AlarconDrinovecForstnericLopez2015MC,AlarconDrinovecForstnericLopez2015PLMS}. 

On a compact bordered Riemann surface $M$, we  
define for every integer $r\geq 1$ the corresponding mapping spaces
$\Lscr^r(M,\C^{2n+1})$ and $\Iscr^r_*(M,\C^{2n}) \subset \Iscr^r(M,\C^{2n})$ by considering 
maps of class $\Cscr^r(M)$ that are holomorphic in the interior $\mathring M=M\setminus bM$;
see Subsec.\ \ref{subs:Legendrian}.
These spaces are complex Banach manifolds (see Theorem \ref{th:Banach}), and hence absolute neighborhood retracts,
and the corresponding maps in the sequence \eqref{eq:fivemaps} are homotopy equivalences (see Remark \ref{rem:php-whe} and Sec.\ \ref{sec:strong}).

We will now explicitly describe the homotopy type of $\Lscr(M,\C^{2n+1})$ and determine its homotopy groups in terms of the homotopy groups of the sphere $\S^{4n-1}$.

A connected open Riemann surface $M$ is homotopy equivalent to a bouquet of circles $\bigvee_{i=1}^\ell \S^1$, where 
$\ell \in\{0,1,\ldots,\infty\}$ is the rank of the free abelian group $H_1(M;\Z)=\Z^\ell$.  For $\ell=0$, we take the bouquet to 
be a point.  The surface $M$ has finite topological type if and only if $\ell$ is finite; then $M$ is biholomorphic to the 
complement of a finite set of points and closed disks in a compact Riemann surface 
(see Stout \cite{Stout1965TAMS}). 

The bouquet $\bigvee_{i=1}^\ell \S^1$ embeds in $M$ as a deformation retract of $M$.  
Hence we have a homotopy equivalence
\[ \Cscr(M,\S^{4n-1}) \to \Cscr(\bigvee_{i=1}^\ell \S^1,\S^{4n-1}). \]
For a space $Y$, let us denote the space $\Cscr(\bigvee_{i=1}^\ell \S^1, Y)$ by $\Lcal_\ell Y$.  
Then $\Lcal_1 Y$ is the free loop space $\Lcal Y$ of $Y$.  It is well known that if we choose a base point $s\in\S^1$, 
then the evaluation map $\Lcal Y\to Y$, $\gamma\mapsto\gamma(s)$, is a fibration whose fibre is the loop space 
$\Omega Y$ of $Y$ \cite[Theorem 10]{Strom1968}.  More generally, taking $s$ to be the common point of the circles 
in the bouquet $\bigvee_{i=1}^\ell \S^1$, $\ell\geq 1$, the evaluation map $\Lcal_\ell Y \to Y$ is a fibration whose fibre 
is $(\Omega Y)^\ell$.

Corollary \ref{cor:whe} now implies the first part of the following result.

%
%
\begin{corollary}\label{cor:loopspace}
Let $M$ be a connected open Riemann surface with $H_1(M;\Z)=\Z^\ell$, $\ell \in\{0,1,\ldots,\infty\}$.  For each $n\geq 1$, 
the spaces $\Lscr(M,\C^{2n+1})$ and $\Lcal_\ell\S^{4n-1}$ are weakly homotopy equivalent.  If $M$ has finite topological 
type, then they are homotopy equivalent.

It follows that $\Lscr(M,\C^{2n+1})$ is path connected and simply connected, and for each $k\geq 2$,
\[  \pi_k(\Lscr(M,\C^{2n+1})) = \pi_k(\S^{4n-1}) \times \pi_{k+1}(\S^{4n-1})^\ell. \]
In particular, $\Lscr(M,\C^{2n+1})$ is $(4n-3)$-connected.
\end{corollary}

\begin{proof}
Recall that $\pi_i(\S^m)=0$ for all $i<m$, and $\pi_m(\S^m)=\Z$.  We must prove the second part of the corollary.  
It is clear for $\ell=0$, so let us assume that $\ell\geq 1$.  Since $Y=\S^{4n-1}$ is simply connected, $\Lcal_\ell Y$ is 
path connected.  Consider the long exact sequence of homotopy groups associated to the fibration $\Lcal_\ell Y\to Y$ 
with fibre $(\Omega Y)^\ell$,
\[ \cdots \to \pi_{k+1}(Y) \to \pi_k((\Omega Y)^\ell) \to \pi_k(\Lcal_\ell Y) \to \pi_k(Y) \to \cdots, \quad k\geq 1, \]
and recall that $\pi_i(\Omega Y)=\pi_{i+1}(Y)$ for all $i\geq 0$.  We see that $\pi_1(\Lcal_\ell Y)=0$.  
The fibration $\Lcal_\ell Y\to Y$ has a section defined by taking a point in $Y$ to the map that takes the whole wedge of 
circles to that point.  Let $k\geq 2$.  The induced sections of the morphisms $\pi_j(\Lcal_\ell Y) \to \pi_j(Y)$ for $j=k$ 
and $j=k+1$ yield a split short exact sequence of abelian groups
\[ 
	0 \to \pi_k((\Omega Y)^\ell) \to \pi_k(\Lcal_\ell Y) \to \pi_k(Y) \to 0, 
\]
demonstrating that $\pi_k(\Lcal_\ell Y) = \pi_k(Y) \times \pi_{k+1}(Y)^\ell$.
\end{proof}

Corollary \ref{cor:loopspace} shows that holomorphic Legendrian 
immersions of an open Riemann surface $M$ into $\C^{2n+1}$ have no homotopy invariants.
Any two such immersions are homotopic through holomorphic Legendrian immersions, and every loop
of Legendrian immersions in $\Lscr(M,\C^3)$ is contractible. The first nontrivial invariant of the space
$\Lscr(M,\C^3)$ is its second homotopy group; see Remark \ref{rem:hyperbolic}. This is very different
from the case of smooth Legendrian knots in a contact $3$-manifold, 
where the basic topological invariants are the rotation number and the 
Thurston-Bennequin number; see e.g.\ \cite{Bennequin1983AST,Eliashberg1993IMRN,FuchsTabachnikov1997T}.

%
%
\begin{remark}\label{rem:hyperbolic}
(a) Theorem \ref{th:immersions} and Corollary \ref{cor:whe} fail for certain other complex contact structures on $\C^{2n+1}$.
Indeed, for any $n\geq 1$, the first author has constructed a Kobayashi hyperbolic complex contact form $\eta$ on $\C^{2n+1}$ \cite{Forstneric2016hyp}. In particular, every 
holomorphic $\eta$-Legendrian map $M\to \C^{2n+1}$ from $M=\C$ or $M=\C_*$ is constant.
Thus, the space $\Lscr_\eta(\C_*,\C^3)=\C^3$ is contractible.
On the other hand, for the $\alpha$-Legendrian maps 
(where $\alpha=dz+xdy$),
\[
	\pi_2(\Lscr_\alpha (\C_*,\C^3)) = \pi_2(\Lcal\, \S^3) = \pi_3(\S^3)=\Z	
\]
by Corollary \ref{cor:loopspace}.  As observed in \cite{Forstneric2016hyp}, the hyperbolic contact forms
$\eta$ constructed there are isotopic to $\alpha$ through a $1$-parameter
family of holomorphic contact forms on $\C^{2n+1}$.

(b) It is easily seen that Corollary \ref{cor:whe} fails if we include
ramified Legendrian maps in the statement. On the other hand, 
it was shown in \cite[Lemma 4.4 and Theorem 5.1]{AlarconForstnericLopez2016Legendrian} 
that every holomorphic Legendrian map of an open Riemann surface to $\C^{2n+1}$ can be approximated 
uniformly on compacts by holomorphic Legendrian embeddings. 
\end{remark}

In conclusion, we observe that holomorphic Legendrian curves in a $3$-dimensional
complex contact manifold are the only smoothly immersed Legendrian surfaces.
Simple examples show that this fails in complex contact manifolds 
of dimension at least $5$. 

%
%
\begin{proposition}\label{prop:complex}
Let $(X,\xi)$ be  a $3$-dimensional complex contact manifold.
If $M$ is a smooth real surface and $F\colon M\to X$ is a smooth
Legendrian immersion, then $F(M)$ is an immersed complex curve in $X$.
Furthermore, $M$ admits the structure of a Riemann surface such that
$F\colon M\to X$ is holomorphic.
\end{proposition}

\begin{proof}
Fix a point $p_0\in M$. By Darboux's theorem, there exist
local holomorphic coordinates $(x,y,z)$ on a neighborhood of the point $F(p_0)\in X$ 
in which the contact structure $\xi$ is given by $\alpha=dz+xdy$. 
Choose smooth local coordinates $(u,v)$ on a neighborhood of $p_0$ in 
$M$ and write $F(u,v)=(x(u,v), y(u,v),z(u,v))$. Then the map $\sigma(u,v)=(x(u,v), y(u,v))$
is an immersion. Differentiation of the equation $dz+xdy=0$ 
gives $dx(u,v)\wedge dy(u,v)=0$ which is equivalent to $x_u y_v-x_v y_u=0$. 
This means that the vectors $\sigma_u = (x_u,y_u)$ and $\sigma_v=(x_v,y_v)$ in $\C^2$
are $\C$-linearly dependent, and hence they span a complex line. Clearly, this line is 
the image of the tangent space $T_{(u,v)} M$ by the differential of $\sigma$ at 
the point $(u,v)$. Finally, since the equation $dz=-xdy$ is $\C$-linear, it follows that 
$dF_{p}(T_{p} M)$ is a complex line in $T_{F(p)} X$ for every point $p\in M$. 

Let $J\colon TX\to TX$ denote the almost complex structure operator induced by the given 
complex structure on $X$. Since $dF_{p}(T_{p} M)$ is a $J$-complex line in $T_{F(p)} X$ 
for every $p\in M$, there exists a unique almost complex structure $J_0\colon TM\to TM$ 
such that $dF_p(J_0 \eta)=J dF_p(\eta)$ for every $p\in M$ and $\eta \in T_p M$. 
The surface $(M,J_0)$ is then a Riemann surface and $F\colon M\to X$ is a 
holomorphic Legendrian immersion.
\end{proof}

\section{Preliminaries}\label{sec:prelim}

%
%
\subsection{Riemann surfaces and mapping spaces}\label{ssec:RS}
For $n\geq 1$, we denote by $|\cdot|$ the Euclidean norm on $\C^n$. 
Given a topological space $K$ and a map $f\colon K\to\C^n$, we define 
\[
	\|f\|_{0,K}:=\sup\{|f(u)|\colon u\in K\}.
\]

Let $M$ be an open Riemann surface. We denote by $\Oscr(M)$ the algebra of all holomorphic functions on $M$. 
If $K$ is a compact subset of $M$, then $\Oscr(K)$ is algebra of all holomorphic functions on 
open neighborhoods of $K$ in $M$, where we identify any pair of functions that agree on some neighborhood of $K$. 
If $K$ is a smoothly bounded compact domain in $M$, then for any integer $r\geq 0$,
we denote by $\Cscr^r(K)$ the algebra of all $r$ times continuously differentiable 
complex valued functions on $K$, and by $\Ascr^r(K)$ the subalgebra of $\Cscr^r(M)$
consisting of all functions that are holomorphic in the interior $\mathring K=K\setminus bK$ of $K$.
We denote by $\|f\|_{r,K}$ the $\Cscr^r$ norm of a function $f\in\Cscr^r(K)$, where the derivatives
are measured with respect to a Riemannian metric on $M$; the choice of the metric
will not be important. The corresponding notation $\Oscr(M)^n$ and $\Ascr^r(K)^n$ and norms 
$\|\cdot\|_{r,K}$ are used for maps $f=(f_1,\ldots,f_n)$ with values in $\C^n$, whose component functions 
$f_j$ belong to the respective function spaces.

A {\em compact bordered Riemann surface} is a compact Riemann surface $M$ whose 
nonempty boundary $bM$ consists of finitely many smooth Jordan curves. 
The interior $\mathring M=M\setminus bM$ of a compact Riemann surface is a {\em bordered Riemann surface}. 
It is classical \cite{Stout1965TAMS} 
that every compact bordered Riemann surface $M$ is conformally equivalent to a smoothly bounded 
compact domain in an open Riemann surface $\wt M$, so the function spaces 
$\Ascr^r(M)$ are defined as above. Note that $\Ascr^r(M)$ is a complex Banach 
algebra for every $r\geq 0$. 

Every bordered Riemann surface $M$ admits smooth closed curves $C_1,\ldots,C_\ell\subset \mathring M$ 
forming a basis of the homology group $H_1(M;\Z)=\Z^\ell$ such that the union $C= \bigcup_{j=1}^\ell C_j$ is 
{\em Runge} in $M$, meaning that the Mergelyan approximation theorem \cite{Mergelyan1951DAN}
holds: every continuous function on $C$ can be uniformly approximated by functions that are holomorphic
on $M$. When $M$ is connected, this holds if and only if $M\setminus C$ has no relatively compact 
connected components.

\subsection{Spaces of Legendrian immersions}\label{subs:Legendrian}
Let $n\in\N=\{1,2,3,\ldots\}$. On the space $\C^{2n+1}$ we use the coordinates $(x,y,z)$ introduced by \eqref{eq:coord}.
To simplify the notation, we often write the standard contact form \eqref{eq:alpha} on $\C^{2n+1}$ in the form
\[
	\alpha = dz+xdy, \quad \text{where}\ \ xdy=  \sum_{j=1}^n x_j \, dy_j.
\]
We identify $\C^{2n}_{(x,y)}$ with the subspace $\{z=0\}\subset \C^{2n+1}$.
Recall (see \eqref{eq:incl})
that $\Iscr(M,\C^{n})$ denotes the space of holomorphic immersions $M\to\C^{n}$,
and $\Iscr_*(M,\C^{2n})$ is the closed subspace of $\Iscr(M,\C^{2n})$ 
consisting of holomorphic immersions $(x,y)\colon M\to\C^{2n}$
for which the holomorphic $1$-form $xdy$ is exact on $M$: the {\em exact holomorphic immersions}. 
The space $\Lscr(M,\C^{2n+1})$ of holomorphic Legendrian immersions
$M\to\C^{2n+1}$ is homeomorphic to $\Iscr_*(M,\C^{2n}) \times\C$ provided $M$ is connected; 
see \eqref{eq:homeo}.

On a compact  bordered Riemann surface $M$ with smooth boundary we introduce the
analogous mapping spaces for any integer $r\geq 1$:
\begin{itemize}
\item $\Iscr^r(M,\C^{n})$ is the space of holomorphic immersions $M\to\C^{n}$ of class $\Ascr^r(M)$;
\vspace{1mm}
\item $\Iscr^r_*(M,\C^{2n})$ is the space of holomorphic immersions $(x,y)\colon M\to\C^{2n}$
of class $\Ascr^r(M)$ for which the holomorphic $1$-form $xdy=  \sum_{j=1}^n x_j \, dy_j$ is exact;
\vspace{1mm}
\item $\Lscr^r(M,\C^{2n+1})$ is the space of immersions $F\colon M\to \C^{2n+1}$ 
of class $\Ascr^r(M)$ such that $F^*\alpha=0$, that is, $F$ is Legendrian with respect to the 
contact form \eqref{eq:alpha}.
\end{itemize}
As in Sec.\ \ref{sec:intro}, when $M$ is connected, the map \eqref{eq:homeo} induces a homeomorphism 
\[
	\Lscr^r(M,\C^{2n+1}) \to \Iscr^r_*(M,\C^{2n}) \times \C.
\]

\subsection{The period map, dominating sprays, and a local structure theorem}\label{ssec:period}

Let $M$ be an open Riemann surface of finite topological type. Let $H_1(M;\Z)=\Z^\ell$ with $\ell\geq 0$.
Pick closed curves $C_1,\ldots,C_\ell\subset M$ forming a Runge homology basis
(see Subsec.\ \ref{ssec:RS}). Let 
\[
	\Pcal=(\Pcal_1,\ldots,\Pcal_\ell) : \Oscr(M)^{2n} \to\C^\ell
\]
be the {\em period map} whose $j$-th component is given by
\begin{equation}\label{eq:periodmap}
	\Pcal_j(x,y)=\int_{C_j} x\, dy,\qquad x,y \in \Oscr(M)^n. 
\end{equation}
Note that $\Pcal(x,y)=0$ if and only if the $1$-form $xdy =  \sum_{i=1}^n x_i \, dy_i$ is exact, and hence
\[
	\Iscr_*(M,\C^{2n}) =\{(x,y)\in \Iscr(M,\C^{2n}) : \Pcal(x,y)=0\}.
\]
If $M$ is a compact smoothly bordered Riemann surface, then \eqref{eq:periodmap} defines a period map 
\begin{equation}\label{eq:P}
	\Pcal : \Ascr^r(M)^{2n} \to\C^\ell,\quad r\in \N,
\end{equation}
and
\[ 
	\Iscr^r_*(M,\C^{2n}) =\{(x,y)\in \Iscr^r(M,\C^{2n}) : \Pcal(x,y)=0\}.
\]

The following lemma provides an important tool used in the proof of Theorem \ref{th:parametric}.
Clearly, the lemma is vacuous if (and only if) $\ell=0$, that is, $M$ is the closed disk $\overline\D$.

%
%
\begin{lemma} 
\label{lem:perioddominatingsprays}
Let $M$ be a compact bordered Riemann surface, and let
$\Pcal$ be the period map \eqref{eq:P} associated to a Runge homology basis of $M$. 
Assume that $P$ is a compact Hausdorff space (a parameter space) and $r\in \N$. 
Given a continuous map $(x,y)\colon P\times M  \to\C^{2n}$ such that 
for every $p\in P$, the map $(x(p,\cdotp),y(p,\cdotp))\colon M\to\C^{2n}$ is 
nonconstant, of class $\Ascr^r(M)$, and its differential is continuous as a function of $(p,u)\in P\times M$,
there exist an integer $N\in\N$ and a continuous map 
$(\tilde x,\tilde y) \colon P\times M \times \C^N \to \C^{2n}$ such that the map
$(\tilde x(p,\cdotp,\cdotp),\tilde y(p,\cdotp,\cdotp)) \colon M\times \C^N \to\C^{2n}$ is of class $\Ascr^r(M\times \C^N)$ for every $p\in P$, its differential is continuous on $P\times M \times \C^N$, and the partial differential
\begin{equation}\label{eq:derivative-period}
	\frac{\di}{\di\zeta}\bigg|_{\zeta=0} \Pcal(\tilde x(p,\cdotp,\zeta),\tilde y(p,\cdotp,\zeta)) \colon \C^N \lra \C^\ell
\end{equation}
is surjective for every $p\in P$. (Here, $\zeta=(\zeta_1,\ldots,\zeta_N)$ are coordinates on $\C^N$.)
\end{lemma}

A map $(\tilde x,\tilde y)$ with surjective differential \eqref{eq:derivative-period} is called a 
{\em period dominating holomorphic spray} of maps $P\times M\to\C^{2n}$
with the core $(\tilde x(\cdotp,\cdotp,0),\tilde x(\cdotp,\cdotp,0)) =(x,y)$.

Note that continuity of a map $(x,y)\colon P\times M  \to\C^{2n}$, which is holomorphic on the interior $\mathring M$ for each $p\in P$, implies continuity of its $M$-derivative of any order on $P\times \mathring M$.  Since the period basis for $M$ is supported in $\mathring M$, 
the lemma holds under this weaker assumption, which already ensures 
continuity of the period map \eqref{eq:derivative-period}.
However, we shall use the lemma in the more general situation when 
$M$ is an admissible set (see Remark \ref{rem:perioddominatingsprays}). 
Since such sets may include arcs, we need the stronger hypothesis 
that the differential is continuous in all variables.

\begin{proof}
Without loss of generality, we assume that the Riemann surface $M$ is connected.  
When $P=\{p\}$ is a singleton, a spray with these properties  was obtained 
in \cite[proof of Theorem 3.3]{AlarconForstnericLopez2016Legendrian}. (We drop $P$ from the notation.)
An inspection of that proof  shows that there exists a spray of this type, with $N=\ell=\mathrm{rank}\, H_1(M;\Z)$, 
such that all but one of its component functions $\tilde x_j,\tilde y_j$ are independent of 
$\zeta\in \C^\ell$. For example, if $y_k$ is nonconstant, 
there is a map $(\tilde x,\tilde y)$ satisfying \eqref{eq:derivative-period} such that for all $u\in M$ and $\zeta\in \C^\ell$ we have
\begin{eqnarray}
	\tilde y(u,\zeta)     &=& y(u), \nonumber \\
	\tilde x_j(u,\zeta)  &=& x_j(u) \quad \text{for} \quad  j\in\{1,\ldots,n\}\setminus\{k\}, \nonumber \\ 
	\tilde x_k(u,\zeta) &=& x_k(u) + \sum_{j=1}^\ell g_j(u)\zeta_j,  \label{eq:X-spray1}
\end{eqnarray}
where the functions $g_1,\ldots, g_\ell\in \Ascr^r(M)$ are chosen 
such that $\int_{C_i} g_j\, dy_k$ approximates the Kronecker symbol $\delta_{i,j}$ for $i,j=1,\ldots, \ell$. 
The approximation can be as close as desired.  One first constructs smooth functions $g_{j}$
on the curves $C_i$ in the homology basis such that $\int_{C_i} g_j\, dy_k = \delta_{i,j}$
and then applies Mergelyan's theorem to obtain functions in $\Ascr^r(M)$.
Similarly, if $x_k$ is nonconstant but $y_k$ is constant, the goal is accomplished by 
letting $\tilde y_k(u,\zeta)=y_k + \sum_{j=1}^\ell g_j(u)\zeta_j$ for suitably chosen functions
$g_1,\ldots, g_\ell\in \Ascr^r(M)$, while the other components of the map 
are independent of $\zeta\in\C^\ell$.

To obtain the parametric case, we observe that the nonparametric
case for a given parameter value $p_0\in P$ automatically satisfies the domination condition 
\eqref{eq:derivative-period} for all points $p$ in an open neighborhood $U\subset P$ of $p_0$. 
Since $P$ is compact, finitely many such neighborhoods $U_1,\ldots, U_m$ cover $P$, 
and it suffices to combine the associated sprays, each with the parameter space $\C^\ell$,
into a single spray with the parameter space $\C^{m\ell}$.
\end{proof}

%
%
\begin{remark}[Admissible sets] \label{rem:perioddominatingsprays}
Lemma \ref{lem:perioddominatingsprays} also holds, with the same proof, 
if $M$ is a compact {\em admissible set} in an open Riemann
surface $\wt M$; see \cite[Definition 5.1]{AlarconForstnericLopez2016MZ}. 
This means that $M=K\cup \Gamma$,
where $K=\bigcup_j K_j$ is a union of finitely many pairwise disjoint, compact, smoothly bounded domains 
$K_j$ in $\wt M$ and $\Gamma=\bigcup_i \Gamma_i$ is a union of finitely many pairwise disjoint smooth arcs or closed 
curves that intersect $K$ only in their endpoints, or not at all, and such that their intersections with the boundary 
$bK$ are transverse. By Mergelyan's theorem 
\cite{Mergelyan1951DAN}, every function $f\in \Ascr^r(M)$, $r\geq 0$,
can be approximated in the $\Cscr^r(M)$-topology by 
functions holomorphic on a neighborhood of $M$. 
If in addition $M$ is Runge ($\Oscr(\wt M)$-convex) in $\wt M$, which holds if and only if the inclusion map $M\hra \wt M$ induces an injective homomorphism $H_1(M;\Z)\hra H_1(\wt M;\Z)$, then the approximation is
possible by functions holomorphic on $\wt M$. 
\end{remark}

An application of Lemma \ref{lem:perioddominatingsprays} and the implicit function theorem give the following structure theorem for the spaces $\Iscr^r_*(M,\C^{2n})$ and $\Lscr^r(M,\C^{2n+1})$.

%
%
\begin{theorem}\label{th:Banach}
Let $M$ be a compact bordered Riemann surface. For every $r\geq 1$, the spaces
$\Iscr^r_*(M,\C^{2n})$ and $\Lscr^r(M,\C^{2n+1})$ are complex Banach manifolds.
\end{theorem}

\begin{proof}
In view of the homeomorphism $\Lscr^r(M,\C^{2n+1}) \to \Iscr^r_*(M,\C^{2n}) \times \C$ induced by the map 
\eqref{eq:homeo}, it suffices to show that $\Iscr^r_*(M,\C^{2n})$ is a closed complex Banach submanifold of 
$\Iscr^r(M,\C^{2n})$, the latter being an open subset of the complex Banach space $\Ascr^r(M)^{2n}$.

Obviously, $\Iscr^r_*(M,\C^{2n})=\{\sigma\in \Iscr^r(M,\C^{2n}): \Pcal(\sigma)=0\}$ is a closed subset of $\Iscr^r(M,\C^{2n})$. 
The period map $\Pcal\colon  \Ascr^r(M)^{2n} \to \C^{\ell}$ is holomorphic. Lemma \ref{lem:perioddominatingsprays} 
(with $P$ a singleton) says that $\Pcal$ has maximal rank $\ell$ at each point $\sigma \in \Ascr^r(M)^{2n}$ that 
represents a nonconstant map. Hence, the conclusion follows from the implicit function theorem.
\end{proof}

It is easily seen that the tangent space to the submanifold $\Iscr^r_*(M,\C^{2n})$ of $\Iscr^r(M,\C^{2n})$ at the point
$\sigma_0=(x_0,y_0)\in \Iscr^r_*(M,\C^{2n})$ equals
\[
	T_{\sigma_0}\Iscr^r_*(M,\C^{2n}) =  \bigl\{\sigma=(x,y)\in \Ascr^r(M)^{2n} : 
	\int_{C_j} xdy_0 + x_0dy=0,\ \ j=1,\ldots,l\bigr\},
\]
where the curves $C_1,\ldots,C_l$ form a basis of $H_1(M;\Z)$.

\section{An application of the convex integration lemma}\label{sec:CI-lemma}

In this section, we establish a key technical result, Lemma \ref{lem:CI}, which will be used
in the proof of Theorem \ref{th:parametric} in order to extend families
of Legendrian immersions across a smooth arc attached to a compact smoothly bounded domain in 
a Riemann surface. 

Let $P$ be a compact Hausdorff space; it will serve as the parameter space. Let 
$\Cscr^{0,1}(P\times [0,1])$ denote the space of all continuous functions $f\colon P\times [0,1]\to\C$, 
considered as a family of paths $f_p=f(p,\cdotp)\colon [0,1]\to \C$ depending 
continuously on  $p\in P$, whose derivative $\dot f_p(s)=df_p(s)/ds$ 
is also continuous in both variables $(p,s)\in P\times [0,1]$. The analogous notation
\[	\Cscr^{0,1}(P\times [0,1],\C^n) = \Cscr^{0,1}(P\times [0,1])^n \]
is used for maps $f=(f_1,\ldots,f_n) \colon P\times [0,1]\to\C^n$. 

We shall need the following lemma.

%
%
\begin{lemma}\label{lem:1dim}
Let $Q\subset P$ be compact Hausdorff spaces, and let 
$f\in \Cscr^{0,1}(P\times [0,1])$ and $h \in \Cscr(P\times [0,1])$  be complex valued functions, 
with $h$ nowhere vanishing. Write $f_p=f(p,\cdotp)$ and similarly for $h$. 
Let $ b\colon P\to \C$ be a continuous function such that
\[	 b(p) =  \int_0^1 f_p(s) h_p(s)\, ds, \quad  p\in Q.  \]
There is a homotopy $f^t\in \Cscr^{0,1}(P\times [0,1])$ $(t\in [0,1])$ satisfying the following conditions:
\begin{itemize}
\item[\rm (i)]     $f^t_p=f_p$ \ for all $(p,t) \in (P\times \{0\}) \cup (Q\times [0,1])$; 
\vspace{1mm}
\item[\rm (ii)]    $f^t_p(s)=f_p(s)$ and $\dot f^t_p(s)=\dot f_p(s)$ for $s=0,1$ and for all $(p,t)\in P\times [0,1]$;
\vspace{1mm}
\item[\rm (iii)]  $\int_0^1 f^1_p(s) h_p(s)\, ds  =  b(p)$  for all $p\in P$.
\end{itemize}    
\end{lemma}

\begin{proof}
This is a parametric version of Gromov's one-dimensional 
{\em convex integration lemma} \cite[Lemma 2.1.7]{Gromov1973IZV}. The basic version of
Gromov's lemma says that for any open connected set $\Omega$ in a Euclidean space $\R^n$
(or in a Banach space), the set of integrals $\int_0^1 f(s)ds$ over all paths $f\colon [0,1]\to \Omega$,
with fixed endpoints $f(0)$ and $f(1)$ in $\Omega$, equals the convex hull of $\Omega$.  
It is a trivial matter to adapt it to arcs of class $\Cscr^1$
with the matching conditions for the derivatives at the endpoints of $[0,1]$. 
For the parametric version we refer to \cite[Theorem 3.4]{Spring2010}. 
The nowhere vanishing function $h$ plays the role of a weight; it would suffice
to assume that $h$ is not identically zero and work on the corresponding subinterval.
\end{proof}

In preparation for the next lemma, we need some additional notation.
Given $z=(z_1,\ldots,z_n),\ w=(z_1,\ldots,z_n)\in \C^n$, we write 
$zw= \sum_{j=1}^n z_j w_j$. We denote by 
\begin{equation}\label{eq:In}
	\Iscr(P\times [0,1],\C^n) \subset \Cscr^{0,1}(P\times [0,1],\C^n) 
\end{equation}
the set of all $f \in  \Cscr^{0,1}(P\times [0,1],\C^n)$ for which 
the derivative $\dot f_p(s)=df_p(s)/ds \in\C^n$ is nowhere vanishing on $(p,s)\in P\times [0,1]$.
We  think of  $f \in \Iscr(P\times [0,1],\C^n)$ as a family of immersed arcs $f_p\colon [0,1]\to\C^n$
depending continuously on the parameter $p\in P$.

The following is the main technical lemma used in the proof of Theorem \ref{th:parametric}.


\begin{lemma}\label{lem:CI}
Let $Q\subset P$ be compact Hausdorff spaces, let $\xi=(f,g)\in \Iscr(P\times [0,1],\C^{2n})$
with $f,g\in \Cscr^{0,1}(P\times [0,1])^n$, and let $\beta\colon P\to \C$ be a continuous function such that
\begin{equation}
\label{eq:alpha_p}
	\beta(p) =  \int_0^1 f_p(s) \dot g_p(s)  ds, \quad  p\in Q.
\end{equation}
Then there exists a homotopy $\xi^t=(f^t,g^t)\in \Iscr(P\times [0,1],\C^{2n})$  
$(t\in [0,1])$ satisfying the following conditions:
\begin{itemize}
\item[\rm (a)]     $\xi^t_p=\xi_p$ \ for $(p,t) \in (P\times \{0\}) \cup (Q\times [0,1])$; 
\vspace{1mm}
\item[\rm (b)]     $\xi^t_p(s)=\xi_p(s)$ and  $\dot \xi^t_p(s)=\dot \xi_p(s)$
for $s=0,1$ and $(p,t)\in P\times [0,1]$;
\vspace{1mm}
\item[\rm (c)]    $\int_0^1 f^1_p(s) \dot g^1_p(s)  ds =\beta(p)$ \ for $p\in P$.
\end{itemize}    
\end{lemma}

In \cite[Lemma 3.1]{ForstnericLarusson2016} we give more precise analogues of Lemmas \ref{lem:1dim} and \ref{lem:CI} 
by controlling the integrals in (iii) and (c) for all $t\in [0,1]$. This can be proved here as well, but is not needed for
the application in the present paper.

\begin{proof}
Since the derivative $\dot \xi_p(s)=(\dot \xi_{p,1}(s),\ldots,\dot \xi_{p,2n}(s))\in\C^{2n}$ 
is nowhere vanishing on $(p,s)\in P\times [0,1]$ and $P$ is compact, 
an elementary argument gives finitely many pairs of compact sets $U_j\subset V_j$ in $P$
$(j=1,\ldots,m)$, with $U_j\subset \mathring V_j$ and $\bigcup_{j=1}^m U_j=P$, 
and pairwise disjoint closed segments $I_1,\ldots, I_m$ contained in $[0,1]$ 
such that for every $j=1,\ldots,m$, there exists an index 
$k=k(j)\in \{1,2,\ldots,2n\}$ such that 
\begin{equation}\label{eq:nonzero}
	\dot \xi_{p,k}(s)\ne 0\quad\text{for all}\ \ s\in I_j\ \text{and}\ p\in V_j. 
\end{equation}

The proof of the lemma proceeds by a finite induction on $j=1,\ldots,m$.
The desired homotopy is obtained as a composition of $m$ homotopies, each supported
on one of the segments $I_1,\ldots,I_m$. We explain the initial step; the subsequent steps are 
analogous. 

Thus, let $j=1$ and let $k=k(1)\in  \{1,2,\ldots,2n\}$ be such that  \eqref{eq:nonzero} 
holds for $j=1$. Suppose first that $k\in \{n+1,\ldots,2n\}$. Write $k=n+l$ with $l\in \{1,\ldots,n\}$.  
Recall that $\xi=(f,g)$ where $f,g\in \Cscr^{0,1}(P\times [0,1])^n$. 
Then \eqref{eq:nonzero}  means that the function $\dot g_{p,l}$ is nowhere vanishing on 
$I_1$ for all $p\in V_1$. Let us define the function $b \colon P\to \C$ by
\begin{equation}\label{eq:betapt}
	b(p)=\beta(p) - \int_{[0,1]\setminus I_1}f_{p,l}(s) \dot g_{p,l}(s)  ds  
	-   \int_0^1 \sum_{\stackrel{i=1}{i\ne l}}^n f_{p,i}(s) \dot g_{p,i}(s)  ds. 
\end{equation}
In view of \eqref{eq:alpha_p} we have that 
\[
	 b(p)=\int_{I_1}f_{p,l}(s) \dot g_{p,l}(s)  ds, \quad  p\in Q. 
\]
We now apply Lemma \ref{lem:1dim} with $Q\subset P$ replaced by the pair of 
parameter sets $V_1 \cap Q \subset V_1$, the interval $[0,1]$ replaced by the segment $I_1$,
with the functions on $I_1$ given by 
\[
	f_p=f_{p,l},\quad h_p=\dot g_{p,l} \quad\text{for $p\in V_1$},
\]
and with the function $b$ given by \eqref{eq:betapt}.
(When applying Lemma  \ref{lem:1dim}, we pay attention to the matching condition (ii) 
at the endpoints of the interval $I_1$).
%
%
This gives a homotopy $f^{t}_{p,l} \in  \Cscr^{0,1}(V_1\times [0,1])$ $(t\in [0,1])$ 
satisfying the following conditions:
\begin{itemize}
\item[\rm (a')]     $f^t_{p,l}=f_{p,l}$ \ for all $(p,t) \in (V_1\times \{0\}) \cup ((Q\cap V_1) \times [0,1])$; 
\vspace{1mm}
\item[\rm (b')]    $f^t_{p,l}(s)=f_{p,l}(s)$ for all $s=[0,1]\setminus I_1$ and $(p,t)\in V_1 \times [0,1]$;
\vspace{1mm}
\item[\rm (c')]    $\int_{I_1} f^1_{p,l}(s) \dot g_{p,l}(s)  ds  = b(p)$  for all $p\in V_1$.
\end{itemize}    
Condition (b') means that the deformation is supported on the segment $I_1$.

Let $\xi^t_p=(f^t_p,g_p)\colon [0,1]\to \C^{2n}$  $(t\in [0,1])$ denote the homotopy whose 
$l$-th component equals $ f^t_{p,l}$ and whose other components agree with the corresponding
components of $\xi_p$. Note that $\xi^t_p$ agrees with $\xi_p$ on $[0,1]\setminus I_1$ 
for all $t\in [0,1]$ and $p\in V_1$, and hence is an immersion 
(since its component $\dot g_{p,l}$ is nowhere vanishing on $I_1$
and $\xi^t_p=\xi_p$ on $[0,1]\setminus I_1$).
Clearly, $\xi^t_p$ satisfies conditions (a) and (b) in Lemma \ref{lem:CI} for  
$(p,t) \in (V_1\times \{0\}) \cup ((Q\cap V_1) \times [0,1])$, 
and it satisfies condition (c) for all $p\in V_1$ in view of the definition \eqref{eq:betapt} of the function $b$.

Pick a continuous function $\chi\colon P\to [0,1]$ such that $\chi=1$ 
on $U_1$ and $\supp\,\chi \subset \mathring V_1$.
Replacing $f^t_p$ by $f^{\chi(p)t}_p$ and $\xi^t_p$ by $\xi^{\chi(p)t}_p$ yields a homotopy, defined for all  $p\in P$,
which satisfies conditions (a) and (b), and it satisfies condition (c) for $p\in U_1$. 

This concludes the first step if $k(1)\in \{n+1,\ldots,2n\}$.
If on the other hand $k=k(1)\in \{1,\ldots,n\}$, we apply the same argument 
with the roles of the components reversed, using the integration by parts formula
\[
	\int_0^1 f_{p,k}(s) \dot g_{p,k}(s) \, ds = f_{p,k}(1)g_{p,k}(1) - f_{p,k}(0)g_{p,k}(0)
	- \int_0^1  g_{p,k}(s) \dot f_{p,k}(s)  \, ds. 
\]
In this case, the assumption is that $\dot f_{p,k}(s)\ne 0$ for all $s\in I_1$ for $p\in V_1$.
The same argument as above gives a homotopy $g^t_{p,k}$, supported on $I_1$,
which achieves condition (c) for all $p\in U_1$. As before, the other components of the map
are kept fixed. 

This concludes the first step of the induction. 

In the second step with $j=2$, we take  as our datum the map $\xi^1\in \Iscr(P\times [0,1],\C^{2n})$ 
(the final map at $t=1$ in the homotopy obtained in step 1). By following the proof of step 1 with the pair 
of parameter sets $Q_2=Q\cup U_1 \subset P$, we find a family of immersions
\[
	\xi^{1,t}_p=(f^{1,t}_p,g^{1,t}_p)   \colon  [0,1]\to \C^{2n}, \quad 
	(p,t) \in P\times [0,1],
\]
satisfying the following conditions: 
\begin{itemize}
\item $\xi^{1,t}_p=\xi^{1}_p$ for $(p,t)\in  (P\times \{0\}) \cup (Q_2\times [0,1])$;
\vspace{1mm}
\item $\xi^{1,t}_p(s)=\xi^{1}_p(s)$ for all $s\in [0,1]\setminus I_2$ and $(p,t)\in P\times [0,1]$;
\vspace{1mm}
\item $\int_0^1 f^{1,1}_p(s) \dot g^{1,1}_p(s) ds =\beta(p)$ \ for all $p\in U_1\cup U_2$.
\end{itemize}
Since the deformation $\xi^{1,t}_p$ is supported on $I_2$ which is disjoint from $I_1$, 
it does not destroy the immersion property of the individual maps $[0,1]\to\C^{2n}$ in
the family. Also, since the deformation is fixed for $p\in Q_2=Q\cup U_1$, 
it does not change the values of the integrals in (c) for $p\in Q_2$, and in addition it achieves 
the correct values for points $p\in U_2$.

We now take $\xi^2=\xi^{1,1}\in \Iscr(P\times [0,1],\C^{2n})$ as the datum in step 3,
let $Q_3=Q_2\cup U_2$, and proceed as before. After $m$ steps of this kind, the proof is complete.
\end{proof}


\section{A parametric Oka principle for Legendrian immersions} \label{sec:hprinciple}

Let $M$ be an open Riemann surface. In this section we prove the parametric Oka principle with approximation
for the inclusion $\Iscr_*(M,\C^{2n}) \hookrightarrow  \Iscr(M,\C^{2n})$ in Theorem \ref{th:immersions}. 

Let $P$ be a compact Hausdorff space. We introduce the following mapping spaces:
\begin{eqnarray*}
	\Iscr(P\times M,\C^{2n}) &=& \{\sigma\in \Cscr(P\times M,\C^{2n}) : 
	\sigma_p \in \Iscr(M,\C^{2n})\ \text{for every}\ p\in P\}; \\
	\Iscr_*(P\times M,\C^{2n}) &=& \{\sigma\in \Iscr(P\times M,\C^{2n}) : 
	\sigma_p\in \Iscr_*(M,\C^{2n}) \ \ \text{for every}\ p\in P\}.
\end{eqnarray*}
Here, $\sigma_p=\sigma(p,\cdotp)\colon M\to \C^{2n}$. 
%
%
Given a compact set $K\subset M$, we write
\[
	\|\sigma\|_{1,P\times K}= \sup_{x\in K} |\sigma_p(x)| + \sup_{x\in K} |d\sigma_p(x)|
\]
where the norm $|d\sigma_p|$ of the differential  is measured with respect to a fixed Hermitian metric on $TM$
(whose precise choice will not be important) and the Euclidean norm on $\C^{2n}$.

%
%
\begin{theorem}\label{th:parametric}
Assume that $M$ is an open Riemann surface, $Q\subset P$ are compact Hausdorff spaces,  
$D\Subset M$ is a smoothly bounded domain whose closure $\bar D$ is $\Oscr(M)$-convex, 
and $\sigma=(x,y)\in \Iscr(P\times M,\C^{2n})$ $(n\ge 1)$ satisfies the following two conditions:
\begin{itemize}
\item[\rm (a)] $\sigma|_{Q\times M} \in \Iscr_*(Q\times M,\C^{2n})$;
\vspace{1mm}
\item[\rm (b)] there is an open set $U\subset M$, with $\bar D\subset U$, such that 
$\sigma|_{P\times U}\in \Iscr_*(P\times U,\C^{2n})$.
\end{itemize}
Given $\epsilon>0$, there is a homotopy $\sigma^t \in \Iscr(P\times M,\C^{2n})$ $(t\in [0,1])$ 
satisfying the following conditions:
\begin{itemize}
\item[\rm (1)]  $\sigma^t_p = \sigma_p$ \ for every $(p,t)\in (P\times \{0\}) \cup (Q\times [0,1])$; 
\vspace{1mm}
\item[\rm (2)]   $\sigma^t|_{P\times D} \in  \Iscr_*(P\times D,\C^{2n})$ for every $t \in [0,1]$;
\vspace{1mm}
\item[\rm (3)]   $\|\sigma^t - \sigma\|_{1,P\times \bar D} <\epsilon$ for every $t\in [0,1]$;
\vspace{1mm}
\item[\rm (4)]   $\sigma^1\in \Iscr_*(P\times M,\C^{2n})$. 
\end{itemize}
\end{theorem}

If a continuous map $\varphi\colon X\to Y$ satisfies the parametric h-principle (without approximation),
then $\varphi$ is a weak homotopy equivalence. Hence, the first part of Theorem \ref{th:immersions}   
is an immediate corollary of Theorem \ref{th:parametric}.

\begin{remark}\label{rem:php-whe}
(a)  The proof of Theorem \ref{th:parametric} gives the analogous result for a compact bordered Riemann surface $M$; 
in this case, the proof is completed in finitely many steps.

(b)  The proof of Theorem \ref{th:parametric} also gives the parametric 
Oka principle with approximation for Legendrian immersions.
However, a minor difference in the proof is explained in the paragraph
following the proof of  Theorem \ref{th:parametric}. It has to do with the fact that the map
$\Lscr(M,\C^{2n+1}) \to \Iscr_*(M,\C^{2n}) \times \C$ (see \eqref{eq:homeo})
is a homeomorphism only when $M$ is connected.  Hence, when extending an exact 
holomorphic immersion $\sigma=(x,y)$ (the projection of a Legendrian immersion $(x,y,z)$)
across a smooth arc $E$ connecting a pair of disjoint domains in $M$, we must ensure that the integral of the
$1$-form $xdy$ on $E$ equals the difference of the values of the last component $z$ at the 
respective endpoints of the arc; in view of \eqref{eq:z-component}, this ensures the correct extension
of the $z$-component.
\end{remark}

\begin{proof}[Proof of Theorem \ref{th:parametric}]
Pick a smooth strongly subharmonic Morse exhaustion function $\rho\colon M\to \R$ and exhaust 
$M$ by sublevel sets 
\[
	D_j=\{u\in M\colon \rho(u)<c_j\}, \quad j\in \N,
\]
where $c_1<c_2<c_3<\ldots$ is an increasing sequence of regular values of $\rho$ chosen such that 
$\lim_{j\to\infty} c_j=\infty$. We may assume that each interval $[c_j,c_{j+1}]$ contains at most one critical 
value of the function $\rho$, and that $D_1$ coincides with the given domain $D$ in 
Theorem \ref{th:parametric}. Let $U_1=U\supset \bar D_1$ be the open neighborhood 
of $\bar D_1$ as in the theorem. 

To begin the induction, set $\epsilon_0=\epsilon$ and  
\[
	\sigma^{t,1} = \sigma|_{P\times U_1} \in  \Iscr_*(P\times U_1,\C^{2n}), \quad t\in [0,1].
\]
We shall inductively find a sequence of open sets $U_j\supset \bar D_j$ in $M$, homotopies 
\[
	\sigma^{t,j} \in  \Iscr(P\times U_j,\C^{2n}), \quad t\in [0,1],\ \  j\in\N
\]	
and  numbers $\epsilon_j>0$ satisfying the following conditions for $j=1,2,3,\ldots$:
\begin{itemize}
\item[$(a_j)$] $\sigma^{t,j}_{p} = \sigma_p|_{U_{j}}$ for every $(p,t)\in (P\times \{0\}) \cup (Q\times [0,1])$; 
\vspace{1mm}
\item[$(b_j)$]  $\sigma^{t,j}|_{P\times D_{1}} \in \Iscr_*(P\times D_{1},\C^{2n})$  for every $t\in  [0,1]$;
\vspace{1mm}
\item[$(c_j)$]  $\|\sigma^{t,j} - \sigma^{t,j-1}\|_{1, P\times \bar D_{j-1}} < \epsilon_j$ for every $t\in  [0,1]$;
\vspace{1mm}
\item[$(d_j)$]  $\sigma^{1,j}|_{P\times D_j}\in \Iscr_*(P\times D_{j},\C^{2n})$;
\vspace{1mm}
\item[$(e_j)$] $\epsilon_j<\epsilon_{j-1}/2$;
\vspace{1mm}
\item[$(f_j)$] If $\tilde \sigma^t \colon P\times \bar D_{j-1} \to\C^{2n}$ satisfies 
$\|\tilde \sigma^{t} - \sigma^{t,j-1}\|_{1, P\times \bar D_{j-1}} < 2\epsilon_j$ for every $t\in  [0,1]$,
then $\tilde \sigma^{t}(p,\cdotp)\colon \bar D_{j-1}  \to\C^{2n}$ is an immersion for every 
$p\in P$ and $t\in  [0,1]$.
\end{itemize}
Conditions $(a_1)$, $(b_1)$ and $(d_1)$ hold by the definition of $\sigma^{t,1}$, 
$(e_1)$ is  fulfilled by choosing $0<\epsilon_1<\epsilon_0/2$, while $(c_1)$ and $(f_1)$ are vacuous. 
Assume for a moment that sequences with these properties exist. 
Conditions $(c_j)$, $(e_j)$ and $(f_j)$ ensure that the sequence $(\sigma^{t,j})_{j\in\N}$ converges to a limit
\[
	\sigma^t = \lim_{j\to \infty} \sigma^{t,j} \colon P\times M\lra \C^{2n}, \quad t\in [0,1]
\]
such that $\sigma^t_p\colon M\to\C^{2n}$ is a holomorphic immersion for every $p\in P$ and $t\in[0,1]$
and (3) holds. Condition $(a_j)$ ensures that all homotopies $\sigma^{t,j}$ are fixed on the parameter 
set $(P\times \{0\}) \cup (Q\times [0,1])$, which gives (1). 
Condition $(b_j)$ shows that  $\sigma^t_p\colon D \to\C^{2n}$ is an exact holomorphic immersion 
for every $p\in P$ and $t\in [0,1]$, so (2) holds. Condition $(d_j)$ shows that $\sigma^1_p\colon M\to\C^{2n}$ 
is an exact holomorphic immersion for every $p\in P$, which gives (4). 
This shows that the theorem holds if we can construct such a sequence of homotopies.

We now explain the induction. Assume that the quantities satisfying the above conditions 
have been found up to an index $j\in \N$. Then, conditions $(e_{j+1})$ and $(f_{j+1})$
hold provided that the number $\epsilon_{j+1}>0$ is chosen small enough; fix such a number.
We shall now explain how to obtain $\sigma^{t,j+1}$ and $U_{j+1}$ satisfying 
conditions $(a_{j+1})$--$(d_{j+1})$.
We distinguish two topologically different cases:  (a) the noncritical case, and (b) the critical case. 

\vspace{1mm} 

{\em (a) The noncritical case:  $\rho$ has no critical values in $[c_j,c_{j+1}]$.}
In this case, $\bar D_j$ is a deformation retract of $\bar D_{j+1}$. 
(In the critical case considered below, we use the noncritical case 
also for certain noncritical pairs of sets $K\subset L$ defined by another
strongly subharmonic function.) 

Pick a Runge homology basis $\Bcal=\{\gamma_i\}_{i=1}^l$ for $H_1(D_j;\Z)$, that is, such that the union of supports $\bigcup_{i=1}^l |\gamma_i|$ is $\Oscr(D_j)$-convex. Let $\Pcal$ 
denote the associated period map \eqref{eq:periodmap}:
\[
	\Pcal(\sigma) =  \left(\int_{\gamma_i} xdy \right)_{i=1,\ldots,l} \in \C^l,\qquad  \sigma=(x,y) \in \Iscr(D_j,\C^{2n}).
\]
Note that the pair $(\Bcal,\Pcal)$ also applies to the domain $D_{j+1}$ since  
$\bar D_j$ is a deformation retract of $\bar D_{j+1}$. 
Let $\zeta=(\zeta_1,\ldots,\zeta_N)$ denote the coordinates on $\C^N$. Shrinking $U_j\supset \bar D_j$ if necessary,
Lemma \ref{lem:perioddominatingsprays}, applied with the parameter space $P'=P\times [0,1]$,
gives an integer $N\in\N$ and a spray
\[
	\tilde \sigma^t = (\tilde x^t,\tilde y^t) \colon P\times U_j   \times \C^N \to \C^{2n},
	\quad t\in [0,1],
\]
such that the map $\tilde \sigma^t_p = \tilde \sigma^t(p,\cdotp,\cdotp) \colon U_j \times \C^N \to\C^{2n}$
satisfies the following conditions:
\begin{itemize} 
\item[\rm (i)] $\tilde \sigma^t_p$ is holomorphic on $U_j   \times \C^N$ for every $(p,t)\in P\times [0,1]$;
\vspace{1mm}
\item[\rm (ii)] $\tilde \sigma^t_p(\cdotp,0)=\sigma^t_p(\cdotp,0)$ at $\zeta=0\in\C^N$  for every $(p,t)\in P\times [0,1]$; 
\vspace{1mm}
\item[\rm (iii)] the partial differential
\begin{equation}\label{eq:tildesigmat}
	\frac{\di}{\di\zeta}\bigg|_{\zeta=0} \Pcal(\tilde \sigma^t_p(\cdotp, \zeta)) \colon \C^N \lra \C^\ell
\end{equation}
is surjective for every $(p,t)\in P\times [0,1]$. 
\end{itemize}
Furthermore, in view of Mergelyan's theorem \cite{Mergelyan1951DAN},
the functions $g_j$ used in the construction of  $\tilde \sigma^t$ (see \eqref{eq:X-spray1}) 
can be chosen holomorphic on $M$. Since the spray $\tilde \sigma^t$ is linear in $\zeta\in\C^N$ and the core 
$\tilde\sigma^t_p(\cdotp,0)=\sigma^t_p$ is holomorphic on $M$ for all $(p,t)\in (P\times \{0\}) \cup (Q\times [0,1])$,
$\tilde \sigma^t_p$ is holomorphic on $M\times \C^N$ for all $(p,t)\in (P\times \{0\}) \cup (Q\times [0,1])$.
Pick an open relatively compact neighborhood $U_{j+1}\Subset M$ of $\bar D_{j+1}$ which deformation
retracts onto $\bar D_{j+1}$. Since the map $\tilde \sigma^t_p(\cdotp,0)=\sigma^t_p$ is an immersion 
on the respective domain for every $(p,t)\in P\times [0,1]$, we can shrink $U_j$ slightly around $\bar D_j$ and
choose a ball $B\subset \C^N$ around the origin such that 
\begin{itemize} 
\item[\rm (iv)]  $\tilde \sigma^t_p(\cdotp,\zeta)\colon U_j\to \C^{2n}$ 
is an immersion for every $(p,t)\in P\times [0,1]$ and $\zeta\in \bar B$, and 
\vspace{1mm}
\item[\rm (v)]  $\tilde \sigma^t_p(\cdotp,\zeta)\colon \overline U_{j+1} \to \C^{2n}$ is an
immersion for all $(p,t) \in (P\times \{0\}) \cup (Q\times [0,1])$ and $\zeta\in \bar B$. 
\end{itemize}

{\em Claim:}  $\tilde \sigma^t$ can be approximated as closely as desired in the $\Cscr^1$ norm on 
$\bar D_j \times \bar B$, and uniformly in the parameters $(p,t)\in P\times [0,1]$, by a homotopy
\[
	\tau^t  \colon P\times U_{j+1} \times B \to \C^{2n},\quad t\in [0,1],
\] 
satisfying conditions (i)--(v) above and also the following two conditions:
\begin{itemize}
\item $\tau^t(p,\cdotp,\zeta) \colon U_{j+1} \to \C^{2n}$ is a holomorphic immersion 
for every $(p,t)\in P\times [0,1]$ and $\zeta\in B$, and
\vspace{1mm}
\item $\tau^t(p,\cdotp,\cdotp) = \tilde \sigma^t(p,\cdotp,\cdotp)$ for all $(p,t)\in (P\times \{0\}) \cup (Q\times [0,1])$.
\end{itemize}

\begin{proof}[Proof of the claim]
Such $\tau^t$ can be found by following the noncritical case in 
\cite[proof of Theorem 5.3]{ForstnericLarusson2016} when the cone $A$ equals $\C^{2n}$. 
The only difference is that, in the present situation,
the maps $\tilde \sigma^t_p$ depend holomorphically on the additional complex parameter 
$\zeta\in B\subset\C^N$. We outline the main steps and refer to the cited source for the details.

Fix a nowhere vanishing holomorphic $1$-form $\theta$ on $M$. Let $d$ denote the exterior
differential on $M$. Consider the family of holomorphic maps 
\begin{equation}\label{eq:phitp}
	\tilde \phi^t_p(\cdotp,\zeta) = d\tilde \sigma^t_p(\cdotp,\zeta)/\theta \colon U_j\to \C^{2n}_*
\end{equation}
for $(p,t)\in P\times [0,1]$ and $\zeta\in \bar B$. Their ranges avoid the origin 
since the maps  $\tilde \sigma^t_p(\cdotp,\zeta)$ are immersions by condition (iv).
Furthermore, for each $(p,t) \in (P\times \{0\}) \cup (Q\times [0,1])$ and $\zeta\in \bar B$, 
the map $\tilde \phi^t_p(\cdotp,\zeta)\colon  \overline U_{j+1}\to \C^{2n}_*$ is holomorphic
on $\overline U_{j+1}$ in view of condition (v).

Let $\Qcal$ denote the period map defined for any map $\phi\colon D_j\to \C^{2n}$ by 
\[
	\Qcal(\phi)=\left(\int_{C_i} \phi\, \theta\right)_{i=1,\ldots,l} \in (\C^{2n})^l.
\]
Here, $\{C_i\}_{i=1}^l$ is a Runge homology basis of $H_1(D_j;\Z)$. We embed the family of maps
\eqref{eq:phitp} as the core of a spray $\phi^t_p(\cdotp,\zeta,w)$ (that is,
$\phi^t_p(\cdotp,\zeta,0)=\tilde \phi^t_p(\cdotp,\zeta)$), depending holomorphically
on another set of parameters $w\in \C^{N'}$ for some integer $N'\in\N$, 
such that the partial differential
\[
	\frac{\di}{\di w}\bigg|_{w=0} \Qcal(\phi^t_p(\cdotp,\zeta,w)) : \C^{N'}\to (\C^{2n})^l
\]
is surjective for every $(p,t)\in P\times [0,1]$ and $\zeta\in B$. 
Such $\Qcal$-period dominating sprays were constructed in \cite[Lemma 5.1]{AlarconForstneric2014IM}; 
see also \cite[Lemma 3.6]{AlarconForstnericCrelle} for the parametric case. 

Fix a ball $B'\subset \C^{N'}$ centered at the origin.
Since $\C^{2n}_*$ is an Oka manifold, the parametric Oka principle with approximation
\cite[Theorem 5.4.4]{Forstneric2011} shows that we can approximate the family 
of holomorphic maps $\phi^t_p\colon U_j\times \bar B\times \bar B' \to \C^{2n}_*$
in the $\Cscr^r$ topology on $\bar D_j \times B\times B'$ by a continuous family of holomorphic maps
\[
	\psi^t_p \colon U_{j+1}\times B\times B'  \to \C^{2n}_*,\quad (p,t)\in P\times [0,1],
\]
such that  $\psi^t_p(\cdotp,\zeta,w)=\phi^t_p(\cdotp,\zeta,w)$ for all $(p,t)\in (P\times \{0\}) \cup (Q\times [0,1])$
and $(\zeta,w) \in B\times B'$.
Assuming that the approximation is close enough, the implicit function
theorem gives a continuous function $w=w(p,t,\zeta)$ on $P\times [0,1]\times \bar B$
with values in $\C^{N'}$ and close to $0$, such that $w$ is holomorphic in $\zeta\in B$, 
vanishes for $(p,t)\in (P\times \{0\}) \cup (Q\times [0,1])$ and $\zeta\in B$, and we have the period vanishing conditions
\begin{equation}\label{eq:Qdom}
	\Qcal\bigl(\psi^t_p(\cdotp,\zeta,w(p,t,\zeta))\bigr) =0 \quad \text{for all}\ (p,t,\zeta)\in P\times[0,1]\times B.   
\end{equation}
Pick an initial point $u_0\in D_j$. It is straightforward to verify that  the family of maps
\[
	\tau^t(p,u,\zeta) = \tilde\sigma^t(p,u_0,\zeta) + \int_{u_0}^u  \psi^t_p(\cdotp,\zeta,w(p,t,\zeta))\, \theta,
	\quad u\in U_{j+1},
\]
then satisfies the claim. (Since $\bar D_j$ is a deformation retract of $U_{j+1}$, the integral is independent 
of the choice of the path in $U_{j+1}$ due to the period vanishing condition \eqref{eq:Qdom}.)
If $D_j$ is disconnected, the same  argument applies on each connected component.
\end{proof}

We continue with the proof of the theorem. Assuming as we may that the approximation 
of $\tilde \sigma^t$ by $\tau^t$ is close enough, the period domination property \eqref{eq:tildesigmat}
of the spray $\tilde \sigma^t$ and the implicit function theorem give a continuous map
\[
	\zeta \colon P\times [0,1] \to B \subset\C^N,
\] 
with values close to $0$ (depending on how close $\tau^t$ is to $\tilde\sigma^t$), such that
\begin{equation}\label{eq:zetavanishes}
	\text{$\zeta$ vanishes on the set $(p,t)\in (P\times \{0\})\cup (Q\times [0,1])$,}
\end{equation}
and the family of holomorphic immersions
\[
	\sigma^{t,j+1}_p = \tau^t(p,\cdotp,\zeta(p,t)) \colon  U_{j+1} \to \C^{2n}
\] 
satisfies the period conditions
\begin{equation}\label{eq:Pjplus1}
	\Pcal(\sigma^{t,j+1}_p)=\Pcal(\sigma^{t,j}_p), \quad (p,t)\in P \times [0,1].
\end{equation}
In view of \eqref{eq:zetavanishes}, $\sigma^{t,j+1}$ satisfies condition $(a_{j+1})$.
Writing $\sigma^{t,j+1}_p=(x^{t,j+1}_p,y^{t,j+1}_p)$, it follows from \eqref{eq:Pjplus1} that for every
loop $C\subset D_{1}$ and for all $(p,t)\in P \times [0,1]$, we have
\[
	\int_C  x^{t,j+1}_p dy^{t,j+1}_p = \int_C  x^{t,j}_p dy^{t,j}_p = 0.
\]
This shows that $\sigma^{t,j+1}$ satisfies condition $(b_{j+1})$. The same argument for
loops $C\subset D_{j+1}$ and $t=1$ shows that $(d_{j+1})$ holds. (Note that it suffices
to verify the period vanishing condition for loops in $\bar D_{j}$, which is a deformation retract of $\bar D_{j+1}$.)
Finally, condition $(c_{j+1})$ holds if the approximations are close enough.

This completes the inductive step in the noncritical case.

\vspace{1mm}

{\em (b) The critical case: $\rho$ has a (unique, Morse) critical point in $D_{j+1}\setminus \bar D_j$.} 
In this case, $\bar D_{j+1}$  deformation retracts onto a compact set 
of the form $S=\bar D_j \cup E$, where $E$ is a smooth embedded arc contained in 
$D_{j+1}\setminus \bar D_j$, except for its endpoints which lie in $bD_j$. We may assume that $E$ 
intersects $bD_j$ transversely at both endpoints. Hence, $S$ is an {\em admissible Runge set}
in $D_{j+1}$ (see Remark \ref{rem:perioddominatingsprays} and 
\cite[Definition 5.1]{AlarconForstnericLopez2016MZ}). 

There are two topologically different cases to consider.

{\em Case 1:}  the arc $E$ closes inside the domain $D_j$ to a Jordan curve $C$ such that 
$E=C\setminus D_j$. This happens when the endpoints of $E$ belong to the same 
connected component of $\bar D_j$. In this case, $H_1(D_{j+1};\Z)=H_1(D_{j};\Z)\oplus \Z$
where $C$ represents the additional generator.

\vspace{1mm}

{\em Case 2:} the endpoints of the arc $E$ belong to different connected components of $\bar D_j$.
In this case, no new element of the homology basis appears.

\vspace{1mm}

We begin with case 1. Let $C$ be a smooth Jordan curve in $M$ such that $E=C\setminus D_j$.
Recall that $\sigma=(x,y) \in \Iscr(P\times M,\C^{2n})$ is the given map in the theorem, and 
$\sigma^{t,j}=(x^{t,j},y^{t,j}) \in \Iscr(P\times U_j,\C^{2n})$ is a homotopy from the $j$-th step. 
After shrinking the neighborhood $U_j$ around $\bar D_j$ if necessary, we can extend 
$\sigma^{t,j}$ from $P\times U_j$ to a homotopy 
\[
	\sigma^{t,j} = (x^{t,j},y^{t,j}) \colon P\times (U_j\cup E) \to \C^{2n}, \quad t\in [0,1]
\]
such that $\sigma^{t,j}_p|_E \colon E \to \C^{2n}$ is a $\Cscr^1$ immersion for every $(p,t)\in P\times [0,1]$ and 
\[
	\sigma^{t,j}_p|_E = \sigma_p|_E\quad \text{for all}\ \ (p,t)\in (P\times \{0\})\cup (Q\times [0,1]).
\]
In particular, condition (a) on $\sigma$ (in the theorem) implies
\begin{equation}\label{eq:OKonQ}
	\int_C x^{t,j}_p dy^{t,j}_p =0 \quad \text{for all}\ \ (p,t)\in Q\times [0,1].
\end{equation}

Our goal is to deform the homotopy $\sigma^{t,j}$ (only) on the relative interior of $E$, keeping it fixed 
for the parameter values $(p,t)\in (P\times \{0\})\cup (Q\times [0,1])$, to a new homotopy
(still denoted $\sigma^{t,j}=(x^{t,j},y^{t,j})$) such that at $t=1$ we have
\begin{equation}\label{eq:intCvanishes}
	\int_C x^{1,j}_p dy^{1,j}_p =0 \quad \text{for all}\ \ p\in P.
\end{equation}
This can be done by using Lemma \ref{lem:CI} as follows.
Choose a smooth regular parametrization $\lambda\colon [0,1]\to E$ with $\lambda(0), \lambda(1) \in bD_j$.
Consider the family of immersed arcs $\xi^t_p=(f^t_p, g^{t}_p) \colon [0,1]\to\C^{2n}$ 
for $(p,t)\in P\times [0,1]$ defined by
\begin{equation}\label{eq:xieta}
	\xi^t_p(s)=\sigma^{t,j}_p(\lambda(s)) = \bigl(f^{t}_p(s), g^{t}_p(s)\bigr), \quad s\in [0,1].
\end{equation}
It follows that
\[
	\int_E x^{t,j}_p dy^{t,j}_p = \int_0^1 f^{t}_p(s) \dot g^{t}_p(s)ds.
\]
Define the function $\beta\colon P\to\C$ by
\begin{equation}\label{eq:beta}
	\beta(p) = - \int_{C\setminus E}  x^{1,j}_p dy^{1,j}_p, \quad p\in P.
\end{equation}
%
%
We now apply Lemma \ref{lem:CI} to the family $(\xi^t_p)_{p,t}$, the pair of parameter spaces 
\[
	(p,t)\in P'=P\times [0,1],\quad Q'=(P\times \{0\})\cup (Q\times [0,1]),
\]
the function $\beta$ given by \eqref{eq:beta}, taking into account condition \eqref{eq:OKonQ}.
This provides a deformation of $(\xi^t_p)_{(p,t)\in P'}$ through a family of immersions $[0,1]\to \C^{2n}$
of class $\Cscr^1$ (the parameter of the homotopy $\tau\in[0,1]$ shall be omitted)  
such that the homotopy is fixed for $(p,t)\in Q'$, it is fixed near the endpoints of $[0,1]$
for all $(p,t)\in P'$, and the new family obtained at $\tau=1$ satisfies the condition 
\[
	\int_0^1 f^{1}_p(s) \dot g^{1}_p(s)ds =\beta(p), \quad p\in P.
\]
By using the parametrization $\lambda\colon [0,1]\to E$ as in \eqref{eq:xieta}, this provides a homotopy
of the family of immersions $\sigma^{t,j}_p=(x^{t,j}_p,y^{t,j}_p) \colon U_j\cup E \to \C^{2n}$ which is 
fixed on $U_j$ such that the new family satisfies the condition
\begin{equation}\label{eq:integralFG}
	\int_E  x^{1,j}_p dy^{1,j}_p = \int_0^1 f^{1}_p(s) \dot g^{1}_p(s)ds =\beta(p),  \quad p\in P.
\end{equation}
Now, \eqref{eq:intCvanishes} follows immediately from \eqref{eq:beta} and \eqref{eq:integralFG}.

Denote by $\Pcal'$ the period map \eqref{eq:periodmap} with respect to
the homology basis $\Bcal$ of $D_j$ and the additional loop $C$. It follows from the above that
$\Pcal' (\sigma^{1,j}_p) = 0$ for all $p\in P$.

The inductive step can now be completed as in the noncritical case; here is an outline.
By Lemma \ref{lem:perioddominatingsprays} we can embed the family of immersions 
$\sigma^{t,j}_p\colon U_j\cup E\to \C^{2n}$ $((p,t)\in P\times [0,1])$
as the core of a period dominating spray depending 
on an additional set of variables $\zeta\in\C^N$. (The set $U_j$ may shrink around $\bar D_j$.)
Since $\bar D_j\cup E$ is an admissible set in $D_{j+1}$ and a deformation
retract of $\bar D_{j+1}$, we can apply the Mergelyan theorem for 
holomorphic immersions to $\C^{2n}$ to approximate this spray, 
as closely as desired in the $\Cscr^1$-topology on $\bar D_j\cup E$, by a spray 
consisting of holomorphic immersions from a neighborhood $U_{j+1}\subset M$ 
of $\bar D_{j+1}$ into $\C^{2n}$. As in the proof of the noncritical case,
replacing the parameter $\zeta$ by a suitably chosen function $\zeta(p,t)$ with values in $\C^N$ 
and close to $0$ gives a homotopy $\sigma^{t,j+1} \in \Iscr(P\times U_{j+1},\C^{2n})$ 
satisfying conditions $(a_{j+1})$--$(d_{j+1})$.

This completes the induction step in case 1 of the critical case (b).

In case 2, the arc $E$ connects two distinct connected components of $\bar D_j$.
We follow the construction in case 1 to obtain an extension
of the family $\sigma^t_p\colon U_j\to\C^{2n}$ across $E$
to a family of immersions $U_j \cup E\to\C^{2n}$; however, there is no need to 
adjust the value of the integral \eqref{eq:integralFG}.
On the other hand, when approximating this family of maps
on $\bar D_j \cup E$ by maps on $U_{j+1} \supset \bar D_{j+1}$,  we still need to use a 
dominating spray as in case 1 in order to keep the period vanishing condition on curves in the 
homology basis $\Bcal$ for $D_j$.
\end{proof}

Returning to Remark \ref{rem:php-whe}, we note that a nontrivial difference appears in the final paragraph 
of the above proof when proving the parametric Oka property for the space of Legendrian immersions.  
Recall that the map $\Lscr(M,\C^{2n+1}) \to \Iscr_*(M,\C^{2n}) \times \C$, given by \eqref{eq:homeo}, 
is a homeomorphism only if $M$ is connected.  When the arc $E$ connects two distinct connected components 
of the set $\bar D_j$, we must ensure the correct value of the integral \eqref{eq:integralFG} in order to match 
the $z$-component of the Legendrian map (which is already defined on a neighborhood of $\bar D_j$) near 
the endpoints of $E$.  This can be achieved just like in case 1.


\section{Strong homotopy equivalence for surfaces of finite topological type} \label{sec:strong}

In this section, we complete the proof of Theorem \ref{th:immersions} by showing that if $M$ is a connected 
open Riemann surface of finite topological type, then the inclusion $\Iscr_*(M,\C^{2n}) \hookrightarrow \Iscr(M,\C^{2n})$, 
already known to be a weak homotopy equivalence, is in fact a homotopy equivalence.  It is even the inclusion of a strong 
deformation retract.  We closely follow the proof of a similar result in \cite[Section 6]{ForstnericLarusson2016}, 
which in turn is based on \cite{Larusson2015PAMS}.

Our approach to showing that the weak homotopy equivalence $\Iscr_*(M,\C^{2n}) \hookrightarrow \Iscr(M,\C^{2n})$ is the 
inclusion of a strong deformation retract is to prove that the metrizable spaces $\Iscr_*(M,\C^{2n})$ and $\Iscr(M,\C^{2n})$
are absolute neighborhood retracts (ANR).  Namely, an ANR has the homotopy type of a CW complex, and a weak homotopy
equivalence between CW complexes is a homotopy equivalence.  Hence, if $j:A\hookrightarrow B$ is the inclusion of a 
closed subspace in a metrizable space $B$, both spaces are ANRs, and $j$ is a weak homotopy equivalence, then $j$ is 
a homotopy equivalence.  Moreover, $j$ is a cofibration (in the sense of Hurewicz), so $j$ is the inclusion of a strong 
deformation retract.  For more information on what is involved, we refer to \cite[Section 6]{ForstnericLarusson2016}.  

The space $\Iscr(M,\C^{2n})$ is an open subset of the Fr\'echet space of all holomorphic maps $M\to\C^{2n}$, so it is an ANR.  

To show that the space $\Iscr_*(M,\C^{2n})$ is an ANR, we verify that it satisfies the so-called Dugundji-Lefschetz property.  
Once we have prepared two ingredients for the proof, it proceeds exactly as the proof of 
\cite[Theorem 6.1]{ForstnericLarusson2016}.

First, we note the homeomorphism
\[ \Iscr(M,\C^{2n}) \to \Oscr_0(M,\C_*^{2n})\times \C, \qquad \sigma\mapsto(d\sigma/\theta, \sigma(p)), \]
where $\Oscr_0(M,\C_*^{2n})$ is the space of holomorphic maps $M\to\C_*^{2n}$ with vanishing 
periods, $\theta$ is a nowhere vanishing holomorphic $1$-form  on $M$,
and $p\in M$ is a chosen base point.  We put together the parametric 
Oka principles with approximation for the inclusion $\Iscr_*(M,\C^{2n}) \hookrightarrow \Iscr(M,\C^{2n})$ 
(Theorem \ref{th:parametric}), for the inclusion $\Oscr_0(M,\C_*^{2n})\hookrightarrow \Oscr(M,\C_*^{2n})$ 
\cite[Theorem 5.3]{ForstnericLarusson2016}, and for the inclusion $\Oscr(M,\C_*^{2n})\hookrightarrow \Cscr(M,\C_*^{2n})$, 
which comes from $\C_*^{2n}$ being an Oka manifold.  This yields the first ingredient: the parametric Oka principle with approximation for the inclusion $\Iscr_*(M,\C^{2n}) \hookrightarrow\Cscr(M,\C_*^{2n})\times\C$.

The second ingredient is the following lemma, which is analogous to \cite[Lemma 6.4]{ForstnericLarusson2016}.  
The proof that $\Iscr_*(M,\C^{2n})$ is an ANR is then so similar to the proof of 
\cite[Theorem 6.1]{ForstnericLarusson2016} that we omit further details.

%
%
\begin{lemma}  \label{lem:contractible}
Let $M$ be an open Riemann surface, let $r\geq 1$ be an integer, and let  $\rho:M\to[0,\infty)$ be a smooth exhaustion function. 
Let $L_0\supset L_1\supset \cdots \supset K$ be compact smoothly bounded domains in $M$ of the form 
$\rho^{-1}([0,c])$, such that $K$ contains all the critical points of $\rho$.  Let $\sigma_0 \in\Iscr_{*}(M,\C^{2n})$ and let 
$W$ be a neighborhood of $\sigma_0|_K$ in $\Iscr_{*}^r(M,\C^{2n})$.  Then there are contractible neighborhoods $C_m$ of 
$\sigma_0|_{L_m}$ in $\Iscr_{*}^r(L_m,\C^{2n})$ such that  $C_m|_{L_{m+1}}\subset C_{m+1}$ and $C_m|_K\subset W$ 
for all $m\geq 0$. 
\end{lemma}

\begin{proof}
Since $K$ contains all the critical points of $\rho$, there is a homology basis
$\Bcal=\{\gamma_i\}_{i=1,\ldots,l}$ of $H_1(M;\Z)$ whose support $|\Bcal| = \bigcup_{j=1}^l |\gamma_j|$ 
is contained in $K$ and is Runge in $M$. Let $\Pcal\colon \Oscr(M,\C^{2n}) \to \C^l$ denote the associated period map \eqref{eq:periodmap}:
\[
	\Pcal(\sigma)= \left(\int_{C_j} x\, dy\right)_{j=1,\ldots,l},  \qquad \sigma=(x,y) \in \Oscr(M,\C^{2n}). 
\]

Fix a map $\sigma_0 \in \Iscr_{*}(M,\C^{2n})$. 
Let $M_0$ be a compact smoothly bounded domain in $M$ (say a sublevel set of $\rho$) 
with the same topology as $M$ and containing $L_0$.
Note that $\Iscr^r(M_0,\C^{2n})$ is an open subset of the complex Banach space $\Ascr^r(M_0,\C^{2n})$.
Pick $\epsilon_0>0$ such that the $\epsilon_0$-ball around $\sigma_0$ in $\Ascr^r(M_0,\C^{2n})$
is contained in $\Iscr^r(M_0,\C^{2n})$.

By Lemma \ref{lem:perioddominatingsprays}, the differential of the period map
$\Pcal:\Ascr^r(M_0,\C^{2n}) \to \C^l$ at $\sigma_0$ is surjective. 
Let us denote it by 
\[	D= d_{\sigma_0}\Pcal : \Ascr^r(M_0,\C^{2n})\lra \C^l. \]
Its kernel
\begin{equation}\label{eq:Lambda0}
	\Lambda_0= \ker D = \{\sigma \in \Ascr^r(M_0,\C^{2n}) : D(\sigma)=0 \}
\end{equation}
is a closed complex subspace of codimension $l$ in $\Ascr^r(M_0,\C^{2n})$;
it is precisely the tangent space to the submanifold $\Iscr_{*}^r(M,\C^{2n})$ at the point $\sigma_0$.
Pick $h_1,\ldots,h_{l} \in \Ascr^r(M_0,\C^{2n})$ such that the vectors $D(h_1),\ldots,D(h_l)\in \C^l$ span $\C^l$; then 
\[
	\Ascr^r(M_0,\C^{2n})=\Lambda_0\oplus \span_\C \{h_1,\ldots,h_{l}\}.
\]

Note that the period map $\Pcal(\sigma)$ is defined whenever the domain $L$ of $\sigma$ contains the support
$|\Bcal|$ of the homology basis. Hence, the map $D=d_{\sigma_0}\Pcal$ is well defined on 
$\Cscr^r(L,\C^{2n})$ whenever $|\Bcal| \subset L \subset M_0$. Taking $L=|\Bcal|$, it follows that the complex 
Banach space  $\Cscr^r(|\Bcal|,\C^{2n})$ decomposes as a direct sum of closed complex Banach subspaces 
\begin{equation}\label{eq:split}
	\Cscr^r(|\Bcal|,\C^{2n}) = \ker D|_{\Cscr^r(|\Bcal|,\C^{2n})} \oplus 
	\span_\C \{h_1|_{|\Bcal|},\ldots,h_{l}|_{|\Bcal|}\} = \Lambda \oplus H. 
\end{equation}
By the implicit function theorem for Banach spaces, there are a number $\epsilon_1\in (0,\epsilon_0)$ 
and smooth bounded complex functions $c_1,\ldots,c_l$ on the set 
$\Lambda_{\epsilon_1}=\{\sigma\in \Lambda : \|\sigma\|_{r,|\Bcal|} < \epsilon_1\}$, vanishing at the origin $0\in \Lambda$,
such that for every $\sigma \in \Lambda_{\epsilon_1}$ the map  
\begin{equation}\label{eq:tildeg}
	\tilde \sigma = \sigma_0|_{|\Bcal|} + \sigma + \sum_{j=1}^{l} c_j(\sigma) h_j|_{|\Bcal|} \in \Cscr^r(|\Bcal|, \C^{2n}) 
\end{equation}
satisfies the period vanishing equation $\Pcal(\tilde \sigma)=0$.
Morever, \eqref{eq:tildeg} gives a local representation of the set 
$\{\tilde \sigma \in  \Cscr^r(|\Bcal|, \C^{2n}): \Pcal(\tilde \sigma)=0\}$ 
in a neighborhood of $\sigma_0|_{|\Bcal|}$ as a graph over the affine linear 
subspace $\sigma_0|_{|\Bcal|}+\Lambda \subset \Cscr^r(|\Bcal|, \C^{2n})$.

If $L$ is any smoothly bounded compact set
with $|\Bcal|\subset L \subset M_0$ and $\sigma \in \Ascr^r(L,\C^{2n})$ satisfies $D\sigma=0$ and $\|\sigma\|_{r,L}<\epsilon_1$, then \eqref{eq:tildeg} yields a map 
\[
	\psi_L(\sigma) =  \sigma_0|_L + \sigma + \sum_{j=1}^{l} c_j(\sigma|_{|\Bcal|}) h_j|_{L}  \in \Ascr^r(L,\C^{2n})
\]
such that $\Pcal(\psi_L(\sigma))=0$. Note that $\psi_L(0) =  \sigma_0$. Hence, $\psi_L(\sigma) \in \Iscr^r_*(L,\C^{2n})$
provided that $\|\psi_L(\sigma)-\sigma_0\|_{r,L}< \epsilon_0$; the latter condition is satisfied if $\epsilon_1>0$ is small enough. 
As before, this gives a local representation of the set 
$\{\tilde \sigma \in  \Ascr^r(L, \C^{2n}) : \Pcal(\tilde \sigma)=0\}$ 
in a neighborhood of $\sigma_0|_L$ as a graph over the affine linear 
subspace $\sigma_0|_{L}+\Lambda_0|_L \subset \Ascr^r(L, \C^{2n})$.
Here, $\Lambda_0=\ker D \subset \Ascr^r(M_0,\C^{2n})$ (see \eqref{eq:Lambda0}).

Note that for any compacts $L$ and $L'$ with $|\Bcal| \subset L\subset L'\subset M_0$, we have
\begin{equation}\label{eq:restriction}
	\psi_{L}(\sigma|_L) = \psi_{L'}(\sigma)\big|_L
\end{equation}
for every $\sigma\in \Ascr^r(L',\C^{2n})$ such that $D(\sigma)=0$ and $\|\sigma\|_{r,|\Bcal|} < \epsilon_1$.

Since the functions $c_j$ are bounded on a neighborhood of the origin in $\Lambda$ (see \eqref{eq:split}), 
there is a number $\epsilon\in (0,\epsilon_1)$ such that the set
\[
	C_0 = \bigl\{\psi_{M_0}(\sigma) : \sigma \in \Lambda_0,\ \|\sigma\|_{r, M_0}< \epsilon\bigr\} \subset \Iscr^r_*(M_0,\C^{2n})
\]
is a neighborhood of $\sigma_0|_{M_0}$ in $\Iscr^r_*(M_0,\C^{2n})$. Furthermore, being a smooth graph 
over the ball $\{\sigma \in \Lambda_0 : \|\sigma\|_{r, M_0}<\epsilon\}$ in the Banach space $\Lambda_0$, $C_0$ is contractible. 
Similarly, for every $m\in \N$, the set
\[
	C_m = \bigl\{\psi_{L_m}(\sigma)  : \sigma \in \Ascr^r(L_m,\C^{2n}),\ D(\sigma)=0,\ \|\sigma\|_{r, L_m}< \epsilon \bigr\}   
	\subset \Iscr^r_*(L_m,\C^{2n})
\]
is a contractible neighborhood of $\sigma_0|_{L_m}$ in $\Iscr^r(L_m,\C^{2n})$. 

Taking into account that for any $\sigma\in \Ascr^r(L_m,\C^{2n})$, we have
$\|\sigma\|_{r,L_{m+1}} \le \|\sigma\|_{r,L_{m}}$ by the maximum principle, 
the formula \eqref{eq:restriction} shows that the restriction map associated to 
the inclusion $L_m\supset L_{m+1}$ maps $C_m$ into $C_{m+1}$ for every $m\geq 0$. 
By choosing $\epsilon>0$ small enough, we can also ensure that the restriction map associated
to $L_m\supset K$ maps $C_m$ into a given neighborhood $W$ of $\sigma_0|_{K}$ in $\Iscr_{*}^r(K,\C^{2n})$.
\end{proof}


\subsection*{Acknowledgements}
F.\ Forstneri\v c is supported in part  by research program P1-0291 and Grant J1-7256 from ARRS, 
Republic of Slovenia.  F.~L\'arusson is supported in part by Australian Research Council Grant DP150103442.  
The work on this paper was done at the Centre for Advanced Study at the Norwegian Academy of Science and Letters in 
Oslo in the autumn of 2016.  The authors would like to warmly thank the Centre for hospitality and financial support.
The authors would like to warmly tbank the Centre for hospitality and financial support.
We thank Antonio Al\'arcon and Francisco J.\ L\'opez for many helpful discussions on this topic,
and Jaka Smrekar for his advice on topological issues concerning loop spaces. 


{\bibliographystyle{abbrv}
\bibliography{bib-FF-FL}}

\vskip 5mm

\noindent Franc Forstneri\v c

\noindent Faculty of Mathematics and Physics, University of Ljubljana, Jadranska 19, SI--1000 Ljubljana, Slovenia

\noindent Institute of Mathematics, Physics and Mechanics, Jadranska 19, SI--1000 Ljubljana, Slovenia

\noindent e-mail: {\tt franc.forstneric@fmf.uni-lj.si}

\vskip 0.5cm

\noindent Finnur L\'arusson

\noindent School of Mathematical Sciences, University of Adelaide, Adelaide SA 5005, Australia

\noindent e-mail:  {\tt finnur.larusson@adelaide.edu.au}

\end{document}